\documentclass[12pt,a4paper]{amsart}
\usepackage{fullpage,amssymb}
\usepackage{hyperref}
\newtheorem{theorem}{Theorem}[section]
\newtheorem{lemma}[theorem]{Lemma}
\theoremstyle{remark}
\newtheorem*{remark}{Remark}

\newcommand{\F}{\mathbb{F}}
\newcommand{\Z}{\mathbb{Z}}
\DeclareMathOperator{\ord}{ord}
\DeclareMathOperator{\rad}{rad}
\DeclareMathOperator{\Tr}{Tr}

\begin{document}
\title{Primitive element pairs with a prescribed trace in the cubic extension of a finite field}
\author{Andrew R. Booker}
\address{University of Bristol, England}
\email{\texttt{andrew.booker@bristol.ac.uk}}
\author{Stephen D. Cohen}
\address{University of Glasgow, Scotland}
\email{\texttt{Stephen.Cohen@glasgow.ac.uk}}
\author{Nicol Leong}
\author{Tim Trudgian}
\address{The University of New South Wales Canberra, Australia}
\email{\texttt{nicol.leong@adfa.edu.au}}
\email{\texttt{t.trudgian@adfa.edu.au}}
\thanks{Trudgian was supported by Australian Research Council Future Fellowship
FT160100094.}
\begin{abstract}
We prove that for any prime power $q\notin\{3,4,5\}$, the
cubic extension $\F_{q^3}$ of the finite field $\F_q$ contains a
primitive element $\xi$ such that $\xi+\xi^{-1}$ is also primitive,
and $\Tr_{\F_{q^3}/\F_q}(\xi)=a$ for any prescribed $a\in\F_q$. This
completes the proof of a conjecture of Gupta, Sharma, and Cohen concerning
the analogous problem over an extension of arbitrary degree $n\ge3$.
\end{abstract}
\maketitle

\section{Introduction}
Let $q$ be a prime power and $n$ an integer at least $3$, and let
$\F_{q^n}$ denote a degree-$n$ extension of the finite field $\F_q$.
We say that $(q,n)\in\mathfrak{P}$ if, for any $a\in\F_q$,
we can find a primitive element $\xi\in\F_{q^n}$
such that $\xi+\xi^{-1}$ is also primitive and $\Tr(\xi)=a$. This problem
was considered by Gupta, Sharma, and Cohen \cite{Gupta-Sharma-Cohen}, who
proved a complete result for $n\ge5$. We refer the reader to \cite{Gupta-Sharma-Cohen} for an introduction to similar problems. Cohen and Gupta \cite{Cohen-Gupta}
extended the work of \cite{Gupta-Sharma-Cohen}, providing a complete result for $n=4$ and some
preliminary results for $n=3$. We improved the latter results in
\cite[\S7]{BCLT}, showing in particular that $(q,3)\in\mathfrak{P}$
for all $q\ge8\times10^{12}$. It is a formidable task to try to prove the result for the remaining values of $q$, and, indeed, the computation involved in \cite{Cohen-Gupta} is extensive.

In this paper we combine theory and novel computation to resolve the remaining cases with $n=3$, proving the
following theorem and affirming \cite[Conjecture~1]{Gupta-Sharma-Cohen}.
\begin{theorem}\label{thm:main}
We have $(q,n)\in\mathfrak{P}$ for all $q$ and all $n\ge3$, with the
exception of the pairs $(3,3)$, $(4,3)$ and $(5,3)$.
\end{theorem}

The main theoretical input that we need is the following result,
which Cohen and Gupta term the ``modified prime sieve criterion'' (MPSC).
\begin{theorem}[\cite{Cohen-Gupta}, Theorem~4.1]
Let $q$ be a prime power, and write $\rad(q^3-1)=kPL$, where
$k,P,L$ are positive integers. Define
\[
\delta=1-2\sum_{p\mid P}\frac1p,
\quad\varepsilon=\sum_{p\mid L}\frac1p,
\quad\theta=\frac{\varphi(k)}{k},
\quad\text{and}\quad
C_q=\begin{cases}
2&\text{if }2\mid q,\\
3&\text{if }2\nmid q.
\end{cases}
\]
Then $(q,3)\in\mathfrak{P}$ provided that
\begin{equation}\label{eq:sievehyp}
\theta^2\delta>2\varepsilon
\quad\text{and}\quad
q^{1/2}>
\frac{C_q\bigl(\theta^24^{\omega(k)}(2\omega(P)-1+2\delta)+\omega(L)-\varepsilon\bigr)}
{\theta^2\delta-2\varepsilon}.
\end{equation}
\end{theorem}
In practice we take $k$ to be the product of the first few prime factors of
$q^3-1$ and $L$ the product of the last few. In particular, taking $L=1$
we recover the simpler ``prime sieve criterion'' (PSC),
\cite[Theorem~3.2]{Cohen-Gupta}, in which the hypothesis
\eqref{eq:sievehyp} reduces to
\[
\delta>0
\quad\text{and}\quad
q^{1/2}>C_q4^{\omega(k)}\left(\frac{2\omega(P)-1}{\delta}+2\right).
\]
We will use this simpler criterion in most of what follows.

\section{Proof of Theorem~\ref{thm:main}}
\subsection{Applying the modified prime sieve}
Thanks to \cite[Theorem~7.2]{BCLT}, to complete the proof
of Theorem~\ref{thm:main} for $n=3$ it suffices to check all
$q<8\times10^{12}$.  To reduce this to a manageable list of
candidates, we seek to apply the MPSC. For prime $q<10^{10}$ and
composite $q<8\times10^{12}$ we do this directly with a straightforward
implementation in \texttt{PARI/GP} \cite{PARI}, first trying the PSC,
and then the general MPSC when necessary.

For larger primes $q$ the direct approach becomes too time-consuming,
mostly because of the time taken to factor $q^3-1$. To remedy this we
developed and coded in \texttt{C} the following novel strategy that makes use of
a partial factorisation. Using sliding window sieves we find the complete
factorisation of $q-1$, as well as all prime factors of $q^2+q+1$ below
$X=2^{20}$. Let $u=(q^2+q+1)\prod_{p<X}p^{-\ord_p(q^2+q+1)}$ denote the
remaining unfactored part. If $u<X^2$ then $u$ must be 1 or a prime number, so we have
enough information to compute the full prime factorisation of $q^3-1$
and can apply the PSC directly.

Otherwise, let $\{p_1,\ldots,p_s\}$ be the set of prime
factors of $u$. Although the $p_i$ are unknown to us, we can bound
their contribution to the PSC via
\[
s\le\lfloor\log_X{u}\rfloor
\quad\text{and}\quad
\sum_{i=1}^s\frac1{p_i}\le\frac{\lfloor\log_X{u}\rfloor}{X}.
\]
We then check the PSC with all possibilities for $P$ divisible by
$p_1\cdots p_s$. This sufficed to rule out all primes
$q\in[10^{10},8\times10^{12}]$ in less than a day using one 16-core
machine.

The end result is a list of $46896$ values of $q$ that are not ruled
out by the MPSC, the largest of which is $4708304701$. Of these, $483$
are composite, the largest being $37951^2=1440278401$. We remark that
with only the PSC there would be $87157$ exceptions, so using the MPSC
reduces the number of candidates by 46\%, and reduces the time taken to
test the candidates (see \S\ref{sec:testing}) by an estimated 61\%.  This represents
an instance when the MPSC makes a substantial and not merely an incidental
contribution  to a computation.

\subsection{Testing the possible exceptions}\label{sec:testing}
Next we aim to test each putative exception directly, by exhibiting,
for each $a\in\F_q$, a primitive pair $(\xi,\xi+\xi^{-1})$ satisfying
$\Tr(\xi)=a$. Although greatly reduced from the initial set of all
$q<8\times10^{12}$ from \cite[Theorem~7.2]{BCLT}, the candidate list is
still rather large, so we employed an optimised search strategy based
on the following lemma.
\begin{lemma}\label{lem:xi0}
Let $g\in\F_q^\times$ be a primitive root, let $d\in\Z$, and set
$P=x^3-x^2+g^{d-1}x-g^d\in\F_q[x]$. Suppose $P$ is irreducible,
let $\xi_0=x+(P)$ be a root of $P$ in $\F_q[x]/(P)\cong\F_{q^3}$,
and assume that $\xi_0$ is not a $p$th power in $\F_{q^3}$ for any
$p\mid q^2+q+1$. Then for any $k\in\Z$ such that $\gcd(3k+d,q-1)=1$,
$\xi_k:=g^k\xi_0$ is a primitive root of $\F_{q^3}^\times$ satisfying
$\Tr(\xi_k)=g^k$ and $\Tr(\xi_k^{-1})=g^{-k-1}$.
\end{lemma}
\begin{proof}
Note that $\xi_0$ has trace $1$ and norm $g^d$, so $\xi_k$ has trace
$g^k$ and norm $\xi_k^{q^2+q+1}=g^{3k+d}$. Furthermore,
\[
\xi_0^3-\xi_0^2+g^{d-1}\xi_0-g^d=0\implies
\xi_0^{-3}-g^{-1}\xi_0^{-2}+g^{-d}\xi_0^{-1}-g^{-d}=0,
\]
so $\Tr(\xi_k^{-1})=g^{-k}\Tr(\xi_0^{-1})=g^{-k-1}$.

Let $p$ be a prime dividing $q^3-1$. If $p\mid q^2+q+1$ then
\[
\xi_k^{\frac{q^3-1}{p}}=(g^k\xi_0)^{\frac{q^3-1}{p}}
=g^{k\frac{q^2+q+1}{p}(q-1)}\xi_0^{\frac{q^3-1}{p}}
=\xi_0^{\frac{q^3-1}{p}}\ne1,
\]
since $\xi_0$ is not a $p$th power. On the other hand, if $p\mid q-1$ then
\[
\xi_k^{\frac{q^3-1}{p}}=\xi_k^{(q^2+q+1)\frac{q-1}{p}}
=g^{(3k+d)\frac{q-1}{p}}\ne1,
\]
since $\gcd(3k+d,q-1)=1$. Hence $\xi_k$ is a primitive root.
\end{proof}
\begin{remark}
If $q\equiv1\pmod3$ then the hypotheses of Lemma~\ref{lem:xi0} imply that
$3\nmid d$.
Hence there always exists $k$ such that $\gcd(3k+d,q-1)=1$, and this
condition is equivalent to $\gcd(k+\bar{d},r)=1$, where
$r=\prod_{\substack{p\mid q-1\\p\ne3}}p$ and $3\bar{d}\equiv d\pmod{r}$.
\end{remark}

Thanks to the symmetry between $\xi$ and $\xi^{-1}$, if we find
a $\xi$ that works for a given $g^k$ via Lemma~\ref{lem:xi0}, we also
obtain a solution for $g^{-k-1}$. Furthermore, when $q\equiv1\pmod4$,
$\alpha\in\F_{q^3}^\times$ is primitive if and only if $-\alpha$
is primitive, and thus a solution for $g^k$ yields one for $-g^k$ by
replacing $\xi$ with $-\xi$.  Therefore, to find a solution for every
$a\in\F_q^\times$ it suffices to check $k\in\{0,\ldots,K-1\}$, where
\[
K=\begin{cases}
\lfloor{q/4}\rfloor&\text{if }q\equiv1\pmod4,\\
\lfloor{q/2}\rfloor&\text{otherwise}.
\end{cases}
\]
Note that this does not handle $a=0$, for which we conduct a separate
search over randomly chosen $\xi\in\F_{q^3}$ of trace $0$.

Our strategy for applying Lemma~\ref{lem:xi0} is as follows. First
we choose random values of $d\pmod{q-1}$ until we find sufficiently
many ($2^{10}$ in our implementation) satisfying the hypotheses. (We
allow repetition among the $d$ values, but for some small $q$
there are no suitable $d$, in which case we fall back on a
brute-force search strategy.) For each $d$ we precompute and
store $\bar{d}=d/3\mod{r}$ and $g^{-d}$, so we can quickly compute
$\xi_k+\xi_k^{-1}=a^{-1}g^{-d}(\xi_0^2-\xi_0)+a\xi_0+a^{-1}g^{-1}$
given the pair $(a,a^{-1})=(g^k,g^{-k})$. Then for each $k$ we run
through the precomputed values of $d$ satisfying $\gcd(k+\bar{d},r)=1$,
and check whether $(\xi_k+\xi_k^{-1})^{(q^3-1)/p}\ne1$ for every
prime $p\mid q^3-1$.

Thanks to Lemma~\ref{lem:xi0}, our test for whether $\xi_k$ itself is
a primitive root is very fast (just a coprimality check\footnote{In
fact, since we run through values of $k$ in linear order, we
could avoid computing the gcd by keeping track of $k\mod{p}$ and
$-\bar{d}\mod{p}$ for each prime $p\mid r$, and looking for collisions
between them. However, in our numerical tests this gave only a small
reduction in the overall running time.}), so we save a factor of roughly
$(q^3-1)/\varphi(q^3-1)$ over a more naive approach that tests both
$\xi$ and $\xi+\xi^{-1}$. Combined with the savings from symmetries
noted above, we estimate that the total running time of our algorithm
over the candidate set is approximately $1/15$th of what it would be
with a direct approach testing random $\xi\in\F_{q^3}$ of trace $a$
for every $a\in\F_q$.

We are not aware of any reason why this strategy should fail
systematically, though we observed that for some fields of small
characteristic (the largest $q$ we encountered is $3^{12}=531441$),
$\xi_k+\xi_k^{-1}$ is always a square for a particular $k$. Whenever
this occurred we fell back on a more straightforward randomised search
for $\xi$ of trace $g^k$ and $g^{-k-1}$.

We used \texttt{PARI/GP} \cite{PARI} to handle the brute-force search
for $q\le211$, as well as the remaining composite $q$ with a basic
implementation of the above strategy. For prime $q>211$ we used Andrew
Sutherland's fast \texttt{C} library \texttt{ff\_poly} \cite{ffpoly}
for arithmetic in $\F_q[x]/(P)$, together with an implementation
of the Bos--Coster algorithm for vector addition chains described
in \cite[\S4]{deRooij}. The total running time for all parts was
approximately 13 days on a computer with 64 cores (AMD Opteron
processors running at 2.5~GHz). The same system handles the largest
value $q=4708304701$ in approximately one hour.

\bibliographystyle{amsplain}
\providecommand{\bysame}{\leavevmode\hbox to3em{\hrulefill}\thinspace}
\providecommand{\MR}{\relax\ifhmode\unskip\space\fi MR }
\providecommand{\MRhref}[2]{%
  \href{http://www.ams.org/mathscinet-getitem?mr=#1}{#2}
}
\providecommand{\href}[2]{#2}

\end{document}